\documentclass[preprint,11pt]{elsarticle}
\usepackage{hyperref}

\usepackage[right=2.5cm,left=2.5cm,bottom=2.5cm,top=2.5cm]{geometry}
\usepackage[reqno]{amsmath}
\usepackage[utf8]{inputenc}
\usepackage{amssymb}
\usepackage{amsmath}
\usepackage{adjustbox}
\usepackage{amsthm}
\usepackage{graphicx}
\usepackage{xcolor}
\usepackage{subcaption}
\usepackage{booktabs}
\usepackage{rotating}
\usepackage{multirow}
\usepackage{comment}
\usepackage{rotating}
\usepackage{url}
\usepackage{paralist}
\usepackage{cleveref}
\usepackage{parskip}
\usepackage{mathrsfs}
\usepackage{graphicx}
\usepackage{caption}
\usepackage{algorithm}
\usepackage{algpseudocode}
\usepackage{microtype}
\usepackage{natbib}
\usepackage{soul}
\usepackage{stmaryrd} %
\usepackage{enumitem} %
\usepackage{mathtools} %
\usepackage{lmodern}
\usepackage{pgf}
\usepackage{multirow}
\usepackage{siunitx}
\usepackage{tikz}
\usetikzlibrary{patterns, shapes.multipart, positioning, arrows.meta, backgrounds, fit, calc}
\usepackage{anyfontsize}
\newcommand{\tinysize}{\fontsize{6}{7}\selectfont}
\makeatletter
\newlength{\hatchdistance}
\newlength{\hatchthickness}
\pgfdeclarepatternformonly[\hatchdistance,\hatchthickness]{customNorthEast}
{\pgfqpoint{0pt}{0pt}}
{\pgfqpoint{\hatchdistance}{\hatchdistance}}
{\pgfqpoint{\hatchdistance}{\hatchdistance}}
{
  \pgfsetcolor{\tikz@pattern@color}
  \pgfsetlinewidth{\hatchthickness}
  \pgfpathmoveto{\pgfqpoint{0pt}{0pt}}
  \pgfpathlineto{\pgfqpoint{\hatchdistance}{\hatchdistance}}
  \pgfusepath{stroke}
}
\makeatother

\usepackage{blkarray}

\graphicspath{ {./images/} } 

\def \comma{,}

\newcommand{\items}{\mathcal{I}}
\newcommand{\colors}{\mathcal{C}}
\newcommand{\profit}{p}
\newcommand{\weight}{w}
\newcommand{\capacity}{b}

\newcommand{\bx}{\boldsymbol{x}}

\newcommand{\formKP}{\text{ILP}_\text{KP}}
\newcommand{\formCKP}{\text{ILP}_\text{CKP}}
\newcommand{\ckpLP}{\text{LP}_\text{CKP}}
\newcommand{\kpLP}{\text{LP}_\text{KP}} 

\newcommand{\firstDP}{{\tt DP_1}}
\newcommand{\secondDP}{{\tt DP_2}}

\newtheorem{proposition}{Proposition}
\newtheorem{lemma}{Lemma}%
\theoremstyle{definition}

\usepackage[acronym]{glossaries}
\newacronym{milp}{MILP}{Mixed Integer Linear Programming}
\newacronym{mip}{MIP}{Mixed Integer Programming}
\newacronym{ilp}{ILP}{Integer Linear Programming}

\usepackage{array}

\newcolumntype{L}[1]{>{\raggedright\arraybackslash}p{#1}}  %
\newcolumntype{R}[1]{>{\raggedleft\arraybackslash}p{#1}}   %

\usepackage{todonotes}
\renewcommand{\arraystretch}{0.8}
\allowdisplaybreaks

\journal{ }

\begin{document}

\begin{frontmatter}
\title{The colored knapsack problem: structural properties and exact algorithms}

\author{Fabio Ciccarelli\textsuperscript{a}}
\author{Alexander Helber\textsuperscript{b}} 
\author{Erik M\"uhmer\textsuperscript{b}}

\let\comma,

\affiliation{a Department of Computer, Control and Management Engineering Antonio Ruberti, Sapienza University of Rome, Rome, Italy}

\affiliation{b Chair of Operations Research, RWTH Aachen University, Aachen, Germany}

\begin{abstract}
We introduce and study a novel generalization of the classical Knapsack Problem (KP), called the \textit{Colored Knapsack Problem} (CKP). 
In this problem, the items are partitioned into classes of colors and the packed items need to be ordered such that no consecutive items are of the same color. 
We establish that the problem is weakly NP-hard and propose two exact dynamic programming algorithms with time complexities of $\mathcal{O}(bn^4)$ and $\mathcal{O}(b^2n^3)$, respectively. 
To enhance practical performance, we derive various dominance and fathoming rules for both approaches. 
From a theoretical perspective, we analyze the linear programming relaxation of the natural CKP formulation, proving that an optimal solution exists with at most two fractional items. We also show that the relaxation can be solved in $\mathcal{O}(n)$ time, matching the complexity of the classical KP. 
Finally, we establish a comprehensive benchmark of CKP instances, derived from the Colored Bin Packing Problem. 
Extensive computational experiments demonstrate that the proposed dynamic programming algorithms significantly outperform state-of-the-art MIP solvers on most of these instances.

\end{abstract}

  \begin{keyword}
    {Combinatorial Optimization \sep Knapsack Problem \sep Integer Linear Programming \sep Dynamic Programming \sep Computational Experiments}
  \end{keyword}
\end{frontmatter}

\section{Introduction}\label{sec:intro}

Given a set $\items=\{1,2,\dots,n\}$ of $n$ {\em items}, each having an associated {\em profit} $\profit_i \in \mathbb{Z}_{> 0}$ and {\em weight} $\weight_i \in \mathbb{Z}_{>0}$
($i \in \items$), and a {\em capacity} $\capacity \in \mathbb{Z}_{>0}, b < \sum_{i = 1}^n \weight_i$,
the {\em Knapsack Problem} (KP) asks to select a subset of items $S \subseteq \items$ which has maximum total profit, and a total weight not exceeding $\capacity$. The classical \textit{Integer Linear Programming} (ILP) formulation for the KP reads as follows:
\begin{subequations}\label{form:KP}
\begin{align}
(\formKP)  && \max ~~\sum_{i \in \items} \profit_i \, x_i \label{eq:kp:obj}\\
  &&\sum_{i \in \items} \weight_i \, x_i & \leq \capacity, \label{eq:kp:capacity}\\
  && x_i&\in\{0,1\}, & i \in \items, \label{eq:KP3}
\end{align}
\end{subequations}
where $x_i$ is a binary variable taking value $1$ if and only if item $i \in \items$ is selected. The objective function \eqref{eq:kp:obj} equals the total profit of the selected items, while the \emph{capacity constraint} \eqref{eq:kp:capacity} ensures that the total weight of the selected items does not exceed the knapsack capacity. This problem has been the subject of intensive research during the last decades, see, e.g., the books by \cite{KPP04li} and \cite{MT90}.

\smallskip
We study a variant of the KP in which we are also given a set of $m$ colors $\colors = \{1, 2, \dots, m\}$, and each item $i \in \items$ has a color $\kappa_i \in \colors$.
The {\em Colored Knapsack Problem} (CKP) consists of the KP with the additional condition that there must exist an ordering of the selected items such that no two consecutive items share the same color.
\citet{BorgesSM24} showed that this is always possible as long as the number of items of each color in the knapsack is less than or equal to that of all other colors together, plus one.
When each item is assigned a distinct color (i.e., there is no pair of distinct items $ i, j \in \items $ such that $ \kappa_i = \kappa_j $), the CKP reduces to the KP. Consequently, the CKP inherits the $\mathcal{NP}$-hardness of the KP; in particular, it is at least weakly $\mathcal{NP}$-hard.
Note that, while for the KP it can be assumed that each item has positive profit, in the CKP items of zero or even negative profit may be packed to respect the color condition in an optimal solution. 
We therefore allow $\boldsymbol \profit \in \mathbb{Z}^n$ for instances of the CKP.

For each color $c \in \colors$, let $\items_c := \{i \in \items \mid \kappa_i = c\}$ be the set of the items of color $c$ and $\bar{\items}_c = \items \setminus \items_c$ the set of the items of all other colors. 
Then, the following is a natural ILP formulation for the CKP, denoted as $\formCKP$:
\begin{subequations}\label{form:CKP}
\begin{align}
(\formCKP)  && \max ~~\sum_{i \in \items} \profit_i \, x_i \label{eq:ckp:obj}\\
  &&\sum_{i \in \items} \weight_i \, x_i & \leq \capacity, \label{eq:ckp:capacity}\\
  &&\sum_{i \in \items_c}x_i - \sum_{i \in \bar{\items}_c}x_i & \leq 1, & c \in \colors, \label{eq:ckp:colors}\\
  && x_i&\in\{0,1\}, & i \in \items. \label{eq:ckp:integrality}
\end{align}
\end{subequations}
The objective function \eqref{eq:ckp:obj} and Constraint \eqref{eq:ckp:capacity} are identical to the objective function and the capacity constraint of $\formKP$, respectively. Constraints \eqref{eq:ckp:colors}, which we refer to as \emph{color constraints}, are instead specific to the CKP and enforce that, for each color $c \in \colors$, the number of items of color $c$ in the knapsack cannot exceed the number of selected items of all other colors, plus one. 

We now introduce some notation and terminology that will be useful in the remainder of the paper. For any subset of items $S \subseteq \items$, let $k_c(S) := |S \cap \items_c|$ be the number of items of color $c$ in $S$. 
We call a color $c$ \emph{dominant} in $S$ if it has the maximum cardinality among all colors, i.e., $c \in \arg\max_{c' \in \colors} \{k_{c'}(S) \}$. 
Note that a set $S$ may admit multiple dominant colors.
A dominant color $c$ is also \emph{critical} for $S$ if
$
k_c(S) = |S \setminus \items_c| + 1.
$
We observe that there can be at most one critical color in any subset of items $S \subseteq \items$. Indeed, for a color $c$ to be critical, it must satisfy $k_c(S) > |S|/2$, meaning that more than half of the items in $S$ share color $c$. Clearly, this condition cannot hold for more than one color. 

The following example, illustrated in Figure \ref{fig:example1}, considers a CKP instance and compares an optimal KP solution, obtained by ignoring item colors, with an optimal CKP solution to that instance. It also shows that, even when a KP-optimal solution is infeasible for the CKP due to the violation of the color condition, there may not exist a CKP-optimal solution in which that color is critical, nor even dominant. Specifically, in the optimal KP solution (a), the subset of selected items is $S=\{1,2\}$ (total profit 23). It is evident that $S$ is not feasible for the CKP, since $k_1(S) > |S \setminus \items_1| + 1$.
On the other hand, the optimal CKP solution (b) selects items $S=\{1,3,4\}$ (total profit 19). In this case, $c=2$ is the critical color.

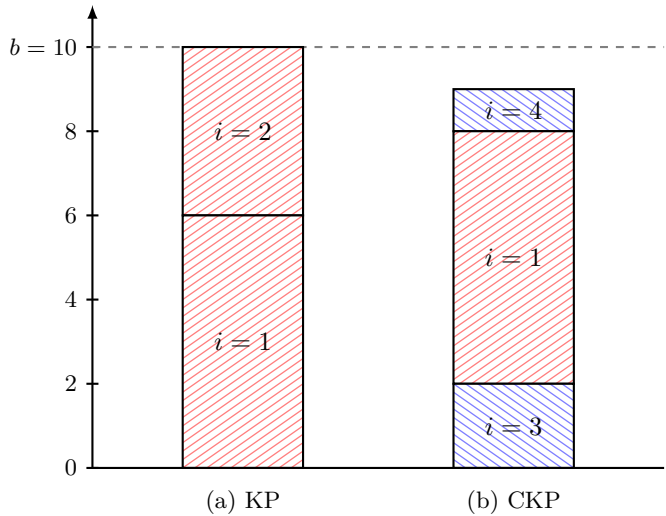
\begin{figure}[h]
    \centering
    \begin{minipage}[c]{0.40\textwidth}
        \centering
        \small 
        \vspace{-0.2cm}
        \text{Instance Data}\\[1.5ex]
        $\capacity = 10$, $\, n = 4$, $\, m = 2$\\[1em]
        \begin{tabular}{lcrrr}
            \toprule
            $i$&& $w_i$ & $p_i$ & $\kappa_i$ \\
            \midrule
            $1$ && $6$ & $15$ & $1$ \\
            $2$ && $4$ & $8$ & $1$ \\
            $3$ && $2$ & $3$  & $2$ \\
            $4$ && $1$ & $1$  & $2$ \\
            \bottomrule
        \end{tabular}
    \end{minipage}
    \begin{minipage}[c]{0.55\textwidth}
        \centering
        \begin{tikzpicture}[y=0.7cm, x=1cm, scale=0.8]

    \draw[->, thick, >=latex] (-1.5, 0) -- (-1.5, 11);
    \foreach \y in {0, 2, 4, 6, 8} {
        \draw[thick] (-1.6, \y) -- (-1.4, \y) node[left, xshift = -4] {\footnotesize $\y$};
    }

    \draw[thick] (-1.6, 10) -- (-1.4, 10) node[left, xshift = -4] {\footnotesize $b=10$};

    \draw[dashed, gray, thick] (-1.5, 10) -- (8, 10);

    \node[] at (1, -0.8) {\footnotesize (a) KP};

    \begin{scope}
        \clip (0, 6) rectangle (2,10);
        \fill[white] (0, 6) rectangle (2,10);
        \foreach \i in {-4,-3.8,...,3} {
            \draw[red!50, line width=0.5pt] (\i,6) -- (\i+4,10);
        }
    \end{scope}
    \draw[thick] (0, 6) rectangle (2,10);
    \draw (1,8) node[] {\small ${i = 2}$};

    \begin{scope}
        \clip (0, 0) rectangle (2,6);
        \fill[white] (0, 0) rectangle (2,6);
        \foreach \i in {-6,-5.8,...,3} {
            \draw[red!50, line width=0.5pt] (\i,0) -- (\i+6,6);
        }
    \end{scope}
    \draw[thick] (0, 0) rectangle (2,6);
    \draw (1, 3) node[] {\small ${i = 1}$};

    \node[] at (5.5, -0.8) {\footnotesize (b) CKP};

    \begin{scope}
        \clip (4.5, 8) rectangle (6.5,9);
        \fill[white] (4.5, 8) rectangle (6.5,9);
        \foreach \i in {3.5,3.7,...,9} {
            \draw[blue!50, line width=0.5pt] (\i,8) -- (\i-1,9);
        }
    \end{scope}
    \draw[thick] (4.5, 8) rectangle (6.5,9);
    \draw (5.5,8.5) node[] {\small ${i=4}$};

    \begin{scope}
    \clip (4.5, 2) rectangle (6.5,8);
    \fill[white] (4.5, 2) rectangle (6.5,8);
    \foreach \i in {-2,-1.8,...,13} {
        \draw[red!50, line width=0.5pt] (\i,2) -- (\i+6,8);
    }
    \end{scope}
    \draw[thick] (4.5, 2) rectangle (6.5,8);
    \draw (5.5,5) node[] {\small ${i=1}$};

    \begin{scope}
    \clip (4.5, 0) rectangle (6.5,2);
    \fill[white] (4.5, 0) rectangle (6.5,2);
    \foreach \i in {2.5,2.7,...,9} {
        \draw[blue!50, line width=0.5pt] (\i,0) -- (\i-2,2);
    }
    \end{scope}
    \draw[thick] (4.5, 0) rectangle (6.5,2);
    \draw (5.5,1) node[] {\small ${i=3}$};

    \draw[-, thick, >=latex] (-1.5, 0) -- (8, 0);

\end{tikzpicture}
    \end{minipage}%
    \hfill 
    
    \caption{Example of a CKP instance. 
    The instance parameters are detailed in the table to the left. 
    The figure to the right shows an optimal KP solution (a), obtained by ignoring the item colors, and an optimal CKP solution (b) for the instance.  
    The vertical axis shows the capacity consumption, and the height of each item corresponds to its weight. Items of color $c=1$ are highlighted in red, while those of color $c=2$ are blue.}
    \label{fig:example1}
\end{figure}

\subsection{Motivation and related problems}
\label{sec:lit}
To the best of our knowledge, the CKP has not been studied yet. 
The closest problem to it, and the one that inspired this work, is the \textit{Colored Bin Packing Problem} (CBPP).
In this problem, items of varying sizes and colors need to be packed in a minimal number of bins of a given size.
For each bin, a color condition as in the CKP is enforced, i.e., the items packed in the bin need to be ordered so that no consecutive items are of the same color.
If the CBPP is solved by branch-and-price, the subproblem is a CKP, motivating the study of this problem.
So we refer to the literature on the CBPP for applications and motivation, except for one application mentioned by~\citet{balogh2013} for which the CKP itself is also relevant:
consider the programming of content, for example deciding which videos to show in the feed of a social media app on a phone screen. 
The content may belong to different categories (educational, funny, music, advertisements, etc.) or be from different content creators and advertisement customers, and a profitable set of videos should be selected to be shown to the user while not showing two videos of the same category consecutively.

The CBPP was first proposed for the two-color case by~\citet{balogh2013}, who investigated the online variant of the problem. 
They followed up with a study of the offline version of the problem, for which they developed approximation algorithms~\citep{balogh_offline_2015}.
The problem was then generalized for the case with more than two colors by~\citet{dosa_colorful_2014}, who also studied the online variant of the problem. 
\citet{BorgesSM24} provides a summary of the results for the online variant and approximation algorithms for the offline variant.

Most relevant for our present work, \citet{BorgesSM24} also propose various exact approaches for solving the offline variant of the CBPP.
First, they introduce a compact (polynomial-size) binary programming model for the problem.
Further, they propose two arc-flow models of pseudo-polynomial size, inspired by the model for the regular Bin Packing Problem of~\citet{ValérioDeCarvalho1999629}, that they solve with a commercial solver.
In these models, a solution is a flow through a network that represents a capacity-expansion of a bin, and arcs represent packing items. 
Each unit of flow represents the packing pattern for a single bin, retrievable by path-decomposition of the solution. 
The benefit of these models is twofold:
first, they improve the quality of the linear programming relaxation bound significantly;
second, they eliminate some symmetry inherent in the compact model by not explicitly modeling the assignment of items to specific bins.
To enforce the color constraints, both models explicitly consider the ordering of the items.
This makes these types of models unpromising for the knapsack variant, since the previously mentioned characterization that simply counts the number of items in each color is much more lightweight. 
They also introduce two exponential size set partitioning models for the CBPP, that they solve with Branch-and-Price using VRPSolver \citep{PessoaSUV20}. 
The pricing problems in those approaches are Colored Knapsack Problems, at least at the root node. 
They model them as Resource-Constrained Shortest Path Problems on two different graphs. 
The first variant uses one resource per color, while the second one requires elementary paths to prevent items from being packing more than once. 
It is worth noticing that, in general, the Resource-Constrained Shortest Path Problem is strongly $\mathcal{NP}$-hard if the number of resources is part of the input (reduction from Binary Programming), or if the paths have to be elementary (reduction from Hamiltonian Path).

Apart from the bin packing variant, some of the plentiful knapsack variants are also similar to our problem.
We refer to~\cite{CACCHIANI2022105692} for an overview of previously studied knapsack variants. 
Among the variants they mention, the most similar in spirit is perhaps the well-known Knapsack Problem with Conflicts, where certain items can not be packed together \citep[see, e.g.,][for a recent exact algorithm for the problem]{ConiglioFS21}.
Nevertheless, this is a much wider-reaching constraint than the very local condition of not packing same-colored items adjacent. 
A recent work that is instead closer to ours is that of \citet{malaguti_algorithms_2025}.
They study the Knapsack Problem with Group Fairness, where the items are partitioned into groups. 
Each item is associated with some resource consumption value and total resource consumption of the items in each group needs to be within some group-specific lower and upper bound. 
In this problem it is also possible to ensure that no group of items is packed more than all others as in our problem, but the number of items for each group needs to be decided apriori, while in our problem the balance must be kept for an arbitrary number of items per color.
They propose a compact ILP model, as well as an extended set-partitioning ILP model.
In the set-partitioning model, they select one packing of items for each group of items.
They make the key observation that there is only a pseudo-polynomial number of relevant packings for each group, allowing them to completely enumerate the packings using a dynamic program and then solve the set-partitioning problem with an off-the-shelf ILP solver.
One of the dynamic programs we propose in this article is based on a similar observation that for each color only a pseudo-polynomial number of packings is relevant; nevertheless, both the dynamic program we use to obtain these packings, as well as the overarching algorithm, is different from their approach due to the differing problem structure.

\subsection{Paper outline and contributions}
The remainder of this paper is organized as follows. 
In \Cref{sec:ILP}, we investigate the theoretical properties of the \textit{Linear Programming} (LP) relaxation of $\formCKP$. We extend the classic results of the KP LP relaxation to this new variant, proving that there always exists an optimal solution to its LP relaxation with at most two fractional items. Furthermore, we analyze the relationship between the LP relaxations of the classical and colored problems, demonstrating that an optimal solution to the LP relaxation of $\formKP$ provides information to identify a tight color constraint for an optimal solution to the LP relaxation of $\formCKP$. Leveraging these insights, we prove that the linear relaxation of $\formCKP$ can be solved in linear time, i.e., $\mathcal{O}(n)$, matching the theoretical complexity of the classical Knapsack Problem.

\smallskip

\Cref{sec:DP} introduces two exact \textit{Dynamic Programming} (DP) algorithms for the CKP, establishing that the problem is weakly NP-hard. The first algorithm, detailed in \Cref{sec:firstDP}, iterates item-by-item, resulting in a time complexity of $\mathcal{O}(b \, n^4)$. The second algorithm, presented in \Cref{sec:secondDP}, adopts a decomposition-like approach that iterates color-by-color, solving an inner Knapsack Problem for each color and an outer recurrence to combine them, with a time complexity of $\mathcal{O}(b^2n^3)$. We further enhance these approaches in \Cref{sec:DPimprovements} by deriving dominance rules to discard dominated states and fathoming rules that prune states which cannot be extended to optimal or feasible solutions.

\smallskip

\Cref{sec:computationals} presents our computational experiments. We introduce a comprehensive benchmark of instances derived from the pricing subproblems of the CBPP, to cover a wide range of structural characteristics. We then evaluate the performance of both dynamic programming approaches, comparing them against the direct solution to the $\formCKP$ formulation using state-of-the-art MIP solvers, in order to demonstrate their practical efficiency.

\smallskip

Finally, \Cref{sec:conclusions} summarizes our main contributions and outlines promising directions for future research.

\section{Properties of the LP relaxation of $\formCKP$}\label{sec:ILP}
The LP relaxation $\kpLP$ of $\formKP$ has some well-known properties, for which we derive analogous properties for the LP relaxation of $\formCKP$, denoted as $\ckpLP$:
\begin{subequations}
\begin{align}
   (\ckpLP)  &  &\max  \sum_{i \in \items} \profit_i x_i, \\
     &    & \sum_{i \in \items} \weight_i x_i        & \leq \capacity, \label{eq:LPcapacity}\\
     && \sum_{i \in \items_c} x_i - \sum_{i \in \bar{\items}_c} x_i & \leq  1, & \forall c \in \colors,  \label{eq:LPcolorbalance}\\
     && x_i & \leq  1, & \forall i \in \items,  \label{eq:LPbound} \\
     && x_i & \geq   0, & \forall i \in \items.
\end{align}
\end{subequations}
Specifically, we will characterize some optimal solutions to $\ckpLP$, show that if an optimal solution to $\kpLP$ violates a color constraint, then there must be an optimal solution to $\ckpLP$ where that color constraint is tight, and use both of these properties to prove that an optimal solution to $\ckpLP$ can be found in linear time. 

For this section, we extend the concepts of dominant and critical colors introduced in \Cref{sec:intro} to solutions of $\ckpLP$. 
Given a solution $\boldsymbol{x} \in [0,1]^n$ to $\ckpLP$, we say a color $c \in \colors$ is dominant if $c \in \arg\max_{c' \in \colors}\sum_{i \in \items_{c'}}x_i$, i.e., if it has the most (fractionally) packed items in solution.
A dominant color $c$ is also critical if $\sum_{i \in \items_c}x_i - \sum_{i \in \bar{\items}_c}x_i = 1$, i.e., if its associated color constraint is tight.

It is well known (see, e.g. \cite{KPP04li}) that it is always possible to construct an optimal solution for $\kpLP$ in which at most one variable takes on a fractional value.
The corresponding item is called the \textit{split item}.
We show an analogous property for $\ckpLP$.
\begin{lemma}\label{lemma:twofracitems}
    There is always an optimal solution $\bx^*$ to $\ckpLP$ in which at most two items are fractionally packed, i.e., \( \big|\{i \mid x_i \in (0,1)\} \big| \leq 2\). 
\end{lemma}
\begin{proof} 
    Consider a feasible and bounded linear program in \( \mathbb{R}^n \).
    Such a problem admits an optimal solution \( \bx^* \) that is an extreme point of the feasible solution polytope.
    Furthermore, at any extreme point in \( \mathbb{R}^n \), at least \( n \) linearly independent constraints are active.
    In $\ckpLP$, constraints consist of one capacity constraint, a family of $m$ color constraints, and bound constraints on the $\bx$ variables.
    We observe that at most one color constraint can be active at any feasible solution to $\ckpLP$: hence, at most two non-bound constraints -- namely, the capacity constraint and at most one color constraint -- can be active at an extreme point.
    It follows that at least  $n - 2$ of the remaining active constraints must be bound constraints.
    Since each bound constraint fixes one variable at either its lower or upper bound, and therefore to an integer value, at most two variables can be fractional at an extreme point solution.
\end{proof}

The two fractional items can be of any color, even the same non-critical color.
We demonstrate this with the following example: consider a CKP instance with \(n = 4\), \(\capacity = 10\), and $m=2$; items $1$ and $2$ have color $c=1$, \(\profit_1 = \profit_2 = 100,\) and \(\weight_1 = \weight_2 =4\); items $3$ and $4$ have instead color $c=2$, $\profit_3 = 2, \profit_4 = 1$, and $\weight_3 = 4, \weight_4 = 1$.
Evidently, in any optimal solution to $\ckpLP$ we have that $x_1=x_2=1$, but to fulfill the color constraint we need that $x_3+x_4 \geq 1$, so that the color constraint associated with color $1$ is not violated. 
Once $x_1$ and $x_2$ have been set to $1$, with the remaining $\capacity -\weight_1 -\weight_2 = 2$ capacity, the best feasible choice is to set $x_3 = \frac{1}{3}$ and $x_4 = \frac{2}{3}$, which fulfills both the color and capacity constraint at equality and packs as much of the third, more profitable item as possible.

\smallskip

Another interesting property of the regular KP is that an optimal solution to its LP relaxation can be found quickly by a combinatorial algorithm, specifically in $\mathcal{O}(n)$ time (see \cite{KPP04li}).
The same result can be extended to $\ckpLP$.
To show this, we need to know if there exists an optimal solution to $\ckpLP$ where a color is critical and, if yes, which it is.
Note that, if no such color exists, we would obtain a feasible solution to $\ckpLP$ simply by solving the LP-relaxation of $\formKP$.
As shown in Figure \ref{fig:example1}, if we solve an instance of the CKP without the color constraints and obtain a solution with a dominant color, it might still be that, considering the color constraints, all optimal CKP solutions have a different dominant color.
In comparison, the LP-relaxation of the problem is much better behaved: indeed, finding an optimal solution to $\kpLP$, thus disregarding the color constraints, provides information about which color is critical in at least one optimal solution to $\ckpLP$, as detailed in the following Lemma.

\begin{lemma}\label{lemma:lp2ckp}
Let \(\bx^*\) denote an optimal solution to the LP-relaxation of $\formKP$ where the color constraint is violated for some color $\tilde c \in \colors$, i.e., \(\sum_{i \in \items_{\tilde c}}x^*_i > \sum_{i \in \bar{\items}_{\tilde c}}x_i^* + 1\).
Then, there always exists an optimal solution $\bx'$ to $\ckpLP$ where $\tilde c$ is critical, i.e., \(\sum_{i \in \items_{\tilde c}}x'_i = \sum_{i \in \bar{\items}_{\tilde c}}x'_i + 1\).
\end{lemma}
\begin{proof}
    Assume there exists an optimal solution $\bx''$ to $\ckpLP$ where $\tilde c$ is not critical.
    Consider the convex combination $\bx(\delta) =\delta \, \bx^* + (1-\delta) \, \bx''$ for $\delta \in [0,1]$. 
    We observe that the color constraint associated to $\tilde c$ is satisfied at $\delta =0$ and violated at $\delta =1$, so by continuity there exists $\overline \delta\in(0,1)$ where it holds at equality and such that $\bx(\overline \delta)$ is feasible for $\ckpLP$. Indeed, both $\bx^*$ and $\bx''$ satisfy all other constraints and so does any convex combination of them. 
    Furthermore, since $\bx''$ cannot have a greater value than $\bx^*$ (i.e., $\boldsymbol{\profit}^\top \bx'' \leq \boldsymbol{\profit}^\top \bx^*$), it holds that
    \begin{equation*}
       \boldsymbol{\profit}^\top \bx(\overline \delta) = \boldsymbol{\profit}^\top \bx'' + \delta \, (\boldsymbol{\profit}^\top \bx^* - \boldsymbol{\profit}^\top \bx'') \ge \boldsymbol{\profit}^\top \bx''.
    \end{equation*}
    This proves that $\bx(\overline \delta)$ is also optimal for $\ckpLP$.
\end{proof}

\Cref{lemma:lp2ckp} establishes that a color $\tilde c \in \colors$ that is critical in at least one optimal solution to $\ckpLP$ can be determined in $\mathcal{O}(n)$ time by solving $\kpLP$, if it exists. This allows to simplify the formulation of $\ckpLP$, discarding the color constraints associated with all colors $c \neq \tilde c$, while turning the color constraint associated to $\tilde c$ into an equality constraint. $\ckpLP$ thus reduces to the following linear program with two structural constraints:
\begin{subequations} \label{form:reducedckpLP}
\begin{align}
     &  &\max  \sum_{i \in \items} \profit_i \, x_i, \\
     &    & \sum_{i \in \items} \weight_i \, x_i        & \leq \capacity,\label{eq:reducedLPcapacity}\\
     && \sum_{i \in \items} g_i \, x_i  & = 1,  \label{eq:reducedLPcolorbalance}\\
     && x_i & \leq  1, & \forall i \in \items,  \\
     && x_i & \geq  0, & \forall i \in \items,
\end{align}
\end{subequations}
where $g_i$ is a parameter equal to $1$ if $\kappa_i = \tilde c$, and $-1$ otherwise.

\smallskip

We now state the main theoretical result of this section, showing that formulation \eqref{form:reducedckpLP} can be solved in linear time.
\begin{proposition}\label{thm:complexityckpLP}
    An optimal solution to the linear program \eqref{form:reducedckpLP} can be computed in $\mathcal{O}(n)$ time.
\end{proposition}

\begin{proof}
    We observe that \eqref{form:reducedckpLP} is a linear program with $n$ variables subject to $2n$ simple bound constraints and exactly two structural constraints (Constraints \eqref{eq:reducedLPcapacity} and \eqref{eq:reducedLPcolorbalance}).
    The dual of this problem involves minimizing a convex piecewise linear function over a space of dimension $d=2$, where the two dual variables correspond to the capacity and the color constraints, respectively.
    Since the dimension $d$ is fixed and independent of the input size $n$, the problem falls into the class of low-dimensional linear programs solvable in linear time via the multidimensional search technique established by \citet{Megiddo1984}.
    Consequently, an optimal solution to \eqref{form:reducedckpLP} can be found in $\mathcal{O}(n)$ time.
\end{proof}

A detailed description of the multidimensional search procedure of \citet{Megiddo1984} is beyond the scope of this paper; we refer the interested reader to the original article for further details, as well as to the works by \citet{ZEMEL1984123} and \citet{megiddo1993}, which discuss applications of this technique to problems sharing structural similarities with ours. 

The results discussed in this section demonstrate that the introduction of color constraints does not increase the computational difficulty of the CKP LP-relaxation, compared to the standard KP. 
The crucial insight provided by Lemma~\ref{lemma:lp2ckp} is that $\ckpLP$ effectively collapses to a linear program with a fixed number of structural constraints (specifically, two) once a critical color (if it exists) is identified by solving the standard KP LP relaxation. This allows to extend the classic linear-time solvability property of $\kpLP$ to the colored variant as well.

\section{Two exact dynamic programming algorithms for solving the CKP}\label{sec:DP}
As discussed in \Cref{sec:intro}, the CKP is at least weakly $\mathcal{NP}$-hard as a generalization of the KP.
In this section, we provide two dynamic programming algorithms that run in pseudo-polynomial time, demonstrating that the problem is in fact also weakly $\mathcal{NP}$-hard.
The first runs in $\mathcal{O}(\capacity \, n^4)$ and the second in $\mathcal{O}(\capacity^2 \, n^3)$, so depending on the instance one or the other might have more promising performance. 

Both algorithms make use of the following observation: to evaluate a solution, we do not actually need to know which colors are dominant, but just the number of items in a dominant color.
So, in both approaches, a state does not keep track of the number of items in each color, but only considers the total number of items, the highest number of items in any color, and for the first algorithm also the number of items in the color of the current item.
This information is always sufficient to determine if a state is feasible, and allows a smaller state space than explicitly tracking the number of items in each color as in one of the approaches proposed by~\citet{BorgesSM24}.

\smallskip

To present our dynamic programming approaches more concisely, we introduce the following notation:
For any item set \(S \subseteq \items\), we denote  by $\profit(S) = \sum_{i \in S} \profit_i$ the total profit of the set and by \(\weight(S) = \sum_{i \in S} \weight_i\) its total weight. 
Moreover, recall that we denote the number of items of color $c \in \colors$ in $S$ by $k_c(S)$ (as introduced in \Cref{sec:intro}). In the following, we will also denote by $\overline{k}_c(S) = \max_{c \in \colors} \, k_c(S)$ the cardinality of a dominant color in $S$.

\subsection{First dynamic program: iterating item by item}
\label{sec:firstDP}
For our first approach, which we refer to as $\firstDP$, we assume the items are sorted by color.
We want to iteratively compute the function $f(i,t,d,a,q)$, that returns the optimal value for the CKP if we restrict the instance to the first $i$ items, pack exactly $t$ items in total, pack exactly $d$ items in the dominant color, pack exactly $a$ items of color $\kappa_i$ (of the last considered item) and use exactly $q$ capacity. 
More formally, we want to compute
\begin{equation}
    f(i,t,d,a,q) = \max\bigg\{\profit(S) \mid S \subseteq \{1,2,\dots, i\}, \, |S| = t, \, \overline{k}(S) = d, \, k_{\kappa_i}(S) = a, \, \weight(S) = q\bigg\}, \label{dp1:goal}
\end{equation}
if such a solution exists.
The optimal solution value for the entire instance can then be determined by finding \(\max_{\, t,d \in \{0,1,\dots,n\}, \, 2  d \leq t + 1, \, a \in \{0,1,\dots,n\}, \, q \in \{0,1, \dots, \capacity \}} f(n,t,d,a,q) \), i.e., finding the largest objective value among solutions that respect the color constraints.

For the sake of presentation, we describe the $\firstDP$ algorithm as it would be implemented, instead of the classic recursive formula, which is hard to understand due to various necessary case distinctions. 
See~\Cref{firstdp} for a pseudo-code description of $\firstDP$.
\begin{algorithm}
    \caption{$\firstDP$ algorithm for the CKP }
    \label{firstdp}
    \begin{algorithmic}[1]
        \Function{$\firstDP$}{$m, n, \boldsymbol \kappa, \boldsymbol\profit, \boldsymbol\weight, \capacity$}
        \State $f(i,t,d,a,q) \gets -\infty$ for all $i,t,d,a \in \{0,\dots,n\}, q \in \{0,\dots,\weight\}$ \Comment{Initialize  possible states}
        \State $f(0,0,0,0,0)  \gets 0 $ \Comment{Initialize starting state}
        \For{$i \in \{1,2,\dots,n\}$}
        \For{$t,d,a,q$ such that $f(i-1,t,d,a,q) > -\infty$} \label{dp1:extendablestates}
        \If{$(i = 1) \lor (\kappa_i = \kappa_{i-1})$} \Comment{Stay in same color}
            \State $a' \gets a$ 
        \Else \Comment{Change color -- reset counter}
            \State $a' \gets 0$ \label{dp1:colorswitch}
        \EndIf
        \State $s_\mathrm{pack} = (i, \, t+1, \, d+\sigma(a',d), \, a' + 1, \, q + \weight_i)$
        \State $f(s_\mathrm{pack}) \gets \max\{f(i-1,t,d,a,q) + \profit_i, \, f(s_\mathrm{pack})\}$ \Comment{Pack item}\label{dp1:packitem}
        \State $f(i,t,d,a',q) \gets \max \{f(i-1,t,d,a,q), f(i,t,d,a',q)\} $ \Comment{Do not pack item}\label{dp1:dontpackitem}
        \EndFor
        \EndFor
        \EndFunction
    \end{algorithmic}
\end{algorithm}

\smallskip

Let $\sigma(a, d): \mathbb{Z}_{\ge 0} \to \{0, 1\}$ be a function returning $1$ if $d \le a$, and $0$ otherwise. 
We start with a state for the empty packing, i.e., $f(0,0,0,0,0) = 0$ and set $f(i,t,d,a,q) = -\infty$ otherwise.
We then iterate through the items and extend all feasible states (Line \ref{dp1:extendablestates}) in two ways: 
\begin{enumerate}
    \item[i)] Packing item $i$ (Line \ref{dp1:packitem}). In this case, the total number of items $t$ increases by one, as does the number of items in the current color $a$. 
    The number of items packed in the currently dominant color increases only if the current color is dominant, i.e., if we have packed exactly $\overline k(S)$ items of the current color in the current solution $S$ associated to the state, and thus $\sigma(a,d) = 1$.  Finally, the weight of the packed items is increased by $w_i$. 
    Then, we set the value of this new state to the maximum between its previously determined value and that of the state we extend plus the profit of the current item, namely $p_i$.
    \item[ii)] Not packing item $i$ (Line \ref{dp1:dontpackitem}). Nothing changes about the packing except that we have considered one more item. The value of this next state is therefore set to the maximum between its previously determined value and that of the current state. 
\end{enumerate}

One special case occurs when we switch from one color to another, i.e., when $\kappa_i \neq \kappa_{i-1}$.
In this case, we set the counter of items packed in the current color back to zero, namely $a' \gets 0$ (Line \ref{dp1:colorswitch}).

\noindent The correctness of $\firstDP$ is formally stated in the following proposition.
\begin{proposition}
The $\firstDP$ algorithm computes the value of $f(i,t,d,a,q)$ for all $i,t,d,a \in \{0,1,\dots,n\}$, $q \in \{0,1,\dots, \capacity\}$ according to~\eqref{dp1:goal}, if a corresponding solution exists, in at most $\mathcal{O}(\capacity \, n^4)$ operations.
\end{proposition}
\begin{proof}
For $i \in \{0,1,\dots,n\}$, $t,d,a \in \{0,1,\dots,n\}$, and $q \in \{0,1,\dots,\capacity\}$, let $f^*(i,t,d,a,q)$ denote the maximum profit obtainable by selecting a subset of the first $i$ items such that exactly $t$ items are selected, the maximum number of selected items of any color equals $d$, exactly $a$ selected items have color $\kappa_i$, and the total weight equals $q$.
If no such subset exists, then we assume $f^*(i,t,d,a,q) = -\infty$.
The $\firstDP$ algorithm maintains a table $f(i,t,d,a,q)$, and we prove that $f(i,t,d,a,q) = f^*(i,t,d,a,q)$ for all $i,t,d,a,q$ by induction on $i$.

\smallskip

For $i=0$, no items are available. The only feasible solution is the empty set, which satisfies $t=d=a=q=0$, and has profit $0$. 
Thus, $f^*(0,0,0,0,0) = 0$, and all other states are infeasible. The algorithm correctly initializes this state and assigns $-\infty$ to all others, so the claim holds for $i=0$.

\smallskip

We now assume that, for some $i \ge 1$, the equality $f(i-1,t,d,a,q) = f^*(i-1,t,d,a,q)$ holds for all $(t,d,a,q)$. We therefore want to show that $\firstDP$ correctly computes $f(i,\cdot)$.
The algorithm iterates over all reachable states $(t,d,a,q)$ at stage $i-1$, i.e., those with $f(i-1,t,d,a,q) > -\infty$, and determines the correct counter $a'$ for the current color $\kappa_i$. 
If $i=1$ or $\kappa_i = \kappa_{i-1}$, then the counter is preserved, so $a' = a$. 
Otherwise, since items are sorted by color, item $i$
is the first occurrence of color $\kappa_i$, and the counter is correctly reset to $a' = 0$. 
From each such state, it considers the two exhaustive and mutually exclusive choices: excluding or including item $i$ among the selected items. 
It follows that: 
i) if $i$ is not selected, the selected set remains unchanged. The total number of selected items, the number of items of each dominant color, and the total weight remain $t$, $d$, and $q$, respectively. Only the counter for the current color is updated to $a'$. The resulting state is therefore $(i,t,d,a',q)$. The algorithm updates $f(i,t,d,a',q)$ by taking the maximum between its current value and the propagated value $f(i-1,t,d,a,q)$; 
ii) if item $i$ is selected, the total number of selected items increases to $t+1$, the total weight to $q+\weight_i$, and the counter of the current color to $a'+1$. 
The value of $d$ must be updated accordingly. 
Since $d$ is the maximum frequency among all colors in the partial solution at stage $i-1$, and the current color count was $a'$, the new dominant frequency is $\max\{d,\, a'+1\}$. 
Because $d \ge a'$, this update can be written as $d + \sigma(a',d)$, where
$\sigma(a',d) = 1$ if $a' = d$ and $0$ otherwise. 
The algorithm thus transitions to the state $(i,\, t+1,\, d+\sigma(a',d),\, a'+1,\, q+\weight_i)$ and updates its value to $f(i-1,t,d,a,q) + \profit_i$, taking the maximum with any previously stored value.

\smallskip

To show that these transitions preserve optimality, consider an optimal solution to a given state $(i,t,d,a,q)$, of value $f^*(i,t,d,a,q)$. 
Item $i$ is either selected or not. 
If item $i$ is not selected, the solution uses only items from $\{1, 2,\dots,i-1\}$ and corresponds to a feasible state at stage $i-1$ with the same profit. 
By the inductive hypothesis, this value is stored in $f(i-1,\cdot)$ and is propagated to stage $i$ by the ``do not pack'' transition.
If item $i$ is instead selected, removing it yields a feasible solution for the first $i-1$ items. 
This partial solution must be optimal for its parameters; otherwise, replacing it with a better one would improve the full solution, contradicting optimality. 
By the inductive hypothesis, its value is correctly stored in $f(i-1,\cdot)$, and the ``pack'' transition adds
$\profit_i$ and updates the state parameters consistently.\\
Since the algorithm considers both cases for every reachable state at stage $i-1$ and updates each destination state by taking the maximum profit over all incoming transitions, it follows that $f(i,t,d,a,q) = f^*(i,t,d,a,q)$ for all $t,d,a,q$.
By induction, this holds for all $i \in \{0,1,\dots,n\}$.

\smallskip

The state space of $\firstDP$ is defined by
$i,t,d,a \in \{0,1,\dots,n\}$ and $q \in \{0,1,\dots,\capacity\}$, and thus contains
$\mathcal{O}(n^4 \capacity)$ states. 
For each reachable state, the algorithm performs a
constant number of operations. 
Consequently, the overall time complexity of $\firstDP$ is $\mathcal{O}(n^4 \capacity)$.
\end{proof}

\medskip

We note that, to actually obtain an optimal CKP solution and not just the optimal solution value, we can store for each state the predecessor state that was extended to reach it, which then allows to backtrack and determine the packed items. 

\subsection{Improvements for $\firstDP$}
\label{sec:DPimprovements}
In this section, we propose some methods to improve the runtime of~\Cref{firstdp} in practice.

First, note that many states permit no feasible solution (i.e., their value is never updated and remains equal to $-\infty$). 
Hence, instead of storing the full dynamic programming table for all values of $(i,t,d,a,q)$, we propose to use a sparse data structure, like a hash map, and only store states that $\firstDP$ actually reaches. 
Second, we observe that some states can actually be merged or fathomed, because they do in fact represent identical partial solutions with respect to feasibility or can not be completed into optimal or feasible solutions.

\subsubsection{Dominance rules} We observe that there may exist pairs of states such that one of these \textit{dominates} the other, meaning that (given the same extension is applied to both the states) one will surely lead to a higher value than the other. It is therefore possible to add some dominance rule that stems from the fact that we so far computed the states as if we solved an exact knapsack problem, i.e., that the capacity constraint has to be fulfilled at equality.

\paragraph{DOMINANCE 1} Consider two states $s_1 = (i,t_1,d_1,a_1,q_1)$ and $s_2 = (i,t_2,t_2,a_2, q_2)$, such that
\begin{equation*}
t_1 \geq t_2, \quad d_1 \leq d_2, \quad a_1 \leq a_2, \quad q_1 \leq q_2 \quad \text{ and } \quad  f(s_1) \geq f(s_2).
\end{equation*}

Then, $s_1$ dominates $s_2$, which can therefore be discarded.  
The reasoning is that every extension that can be applied to $s_2$ that results in a feasible solution can also be applied to $s_1$, since it is never further away from being feasible.
Observe that $s_1$ has no less items packed in total and no more items packed in the dominant and current color, so it is not further away from fulfilling the color constraint than $s_2$.
Additionally, $s_1$ has no less remaining capacity.
Since it also gives at least the same profit, there is no need to keep state $s_2$.

For practical purposes, preliminary experiments showed that it is beneficial to only check this dominance criterion when a color was just finished and all values of $a$ can be set to zero, so significantly less comparisons are necessary. For the same reason, we only check states with $t_1=t_2$, instead of all those with $t_1\geq t_2$.

\paragraph{DOMINANCE 2} 
The next group of state space reductions concerns situations in which two or more states are identical with regards to feasibility of the color constraints.
First, consider the case that we switch to a new color.
A state with $2d \leq t + 1$ is feasible with regards to the color constraint and all previously considered colors cannot become dominant anymore. 
So, at the first item of a new color, we can simply set the value of $d$ in the new states to zero if $2d \leq t+1$ for the state that is being extended.
A variant of this reduction rule can also be applied within a color: if the number of remaining items $\bar{n}$ in the current color is not sufficient to construct an infeasible solution, i.e., $2d \leq t+1$ and $2(a + \bar{n}) \leq t + 1 $, we can set $d = 0$ and $a = 0$. 
We can also set $a=0$ if $a + \bar{n} \leq d$, because the current color cannot become dominant at all. 
By mapping $d$ or $a$ to zero, we merge multiple states and only keep the non-dominated ones.

\subsubsection{Fathoming rules}
In order to reduce the number of states that can be generated during the algorithm execution, it is also crucial to fathom some states based on optimistic bounds regarding the potential completions of the current partial solutions associated to them.
Specifically, we fathom a state if it cannot become feasible or if it cannot achieve a better solution value than the best known primal bound with the remaining available items and capacity.
To do so, we need to know the value
\begin{equation*}
    \bar{\profit}(i,q) =  \max\{\profit(S) \mid S \subseteq \{i,i+1,\dots,n\}, \weight(S) \leq q\},
\end{equation*}
which is the maximum profit achievable with the remaining items starting from item $i$ and with $q$ capacity or less, disregarding the color constraints. We also need to know the value
\begin{equation*}
    \bar{n}(i,q) = \max\{|S| \mid S \subseteq \{i,i+1,\dots,n\}, \weight(S)\leq q\},
\end{equation*}
which is the maximum number of items that can be packed among the remaining ones starting from item $i$, and with $q$ capacity or less, disregarding the color constraints. 
Note that both values can be precomputed in $\mathcal{O}(\capacity \, n)$ operations for all $(i,q)$ by solving a regular KP on the instance with the items ordered in reverse, using dynamic programming. We can therefore derive the following fathoming rules.

\paragraph{FATHOMING 1}
    Let $LB$ denote any primal bound on the optimal CKP solution value.
    Then, if for a state $s = (i,t,d,a,q)$ it holds that if
    \begin{equation*}
    f(s) + \bar{\profit}(i+1,\capacity - q) \leq LB,
    \end{equation*}
    the state $s$ cannot lead to any solution of value greater than $LB$ and, therefore, it can be discarded.

\paragraph{FATHOMING 2}
    Let $i' = \min\{j \in \{i,i+1,\dots,n\} \mid \kappa_j \neq \kappa_i\}$ denote the index of the first item whose color is different from $\kappa_i$.
    If for a state $s = (i,t,d,a,q)$ it holds that
    \begin{equation*}
    2d > \begin{cases}
         t + 1 + \bar{n}(i+1,\capacity - q), \qquad & \text{ if } a < d,\\[1.5ex]
         t + 1 + \bar{n}(i', \capacity - q), & \text{ otherwise,}
    \end{cases}
    \end{equation*}
    then the state $s$ cannot lead to any feasible solution and, therefore, it can be discarded.

\smallskip

We use the information we obtain from determining $\bar{\boldsymbol \profit}$ also in another way: for all $i \in \items$, $q \in \{0,1,\dots, \capacity\}$, we determine the solution leading to the value of $\bar{\profit}(i,q)$; we then compute the total number of items $\bar{t}(i,q)$, the number of items in the dominant color $\bar{d}(i,q)$, and the number of items $\bar{a}(i,q)$ in color $\kappa_i$ for this solution.
Then, the following result holds.

\paragraph{FATHOMING 3}
    If for a state $s = (i,t,d,a,q)$ it holds that
    \begin{equation*}
    \max\left\{a + \bar{a}(i+1,\capacity - q), \, d, \, \bar{d}(i+1,\capacity-q) \right\} \, \leq \, t+\bar{t}(i+1, \, \capacity- q) + 1,
    \end{equation*}
    then the primal bound $LB$ is updated to $f(s) + \bar{p}(i+1, b-q)$, the combined solution is stored, and state~$s$ is fathomed.

\medskip

\subsection{Example for $\firstDP$}
\Cref{fig:firstDPexample} shows the execution of $\firstDP$, including the proposed improvements, on the instance introduced in \Cref{fig:example1}, with $n=4$, $m =2$, and $b=10$.
In the figure, each node represents a generated state identified by the tuple $s = (t, d, a, q)$, with the associated profit $f(s)$ displayed above the tuple.
Solid arrows correspond to the decision of packing the current item, while dashed arrows represent skipping it. Fathomed states are highlighted in red, and report the reason why they have been fathomed below their associated node. The nodes are divided into layers, each corresponding to a different stage of the algorithm (and, therefore, to a different item considered for packing). 
The header of each layer shows the item index, along with its weight and profit in tuple format. Finally, the path highlighted in green indicates the sequence of decisions leading to the found optimal solution.
Note that some of the states are discarded according to the first fathoming rule: as an example, take state $s=(1,1,1,6)$ at Stage 3, which has a value of $15$.
This state cannot be extended to achieve a value greater than $18$ (the current lower (primal) bound) and it is therefore fathomed by bound.
For a visualization of the effect of the second fathoming rule, consider instead the state generated at Stage 2 by packing both items of color $c=1$ (the bottom-most node). 
Although this partial solution yields a high profit, equal to $23$, it is immediately fathomed. Indeed, since the capacity is fully utilized ($q=10$), it is impossible to pack any subsequent items of color $c=2$, required to restore feasibility. Consequently, the condition $2d \le t+1$ cannot be satisfied by any completion, and the state is discarded.
Notice that, for the sake of presentation, we are not considering the effect of the third fathoming rule: indeed, if we applied such rule, the state $s = (1,1,1,6)$, at Stage 2, would be immediately expanded to the optimal one, and discarded accordingly.

\begin{figure}[htbp]
    \centering
    \usetikzlibrary{shapes.multipart}
\usetikzlibrary{positioning}      
\usetikzlibrary{arrows.meta}      
\usetikzlibrary{backgrounds}      
\usetikzlibrary{calc}             

\pgfdeclarelayer{background}
\pgfsetlayers{background,main}

\begin{tikzpicture}[
    font=\scriptsize,
    >=Stealth,
    xscale=3.2, 
    yscale=1.6,
    state/.style={
        rectangle split,
        rectangle split parts=2,
        draw=black!80,
        rounded corners=2pt,
        fill=white,
        align=center,
        inner sep=3pt,
        text width=1.4cm
    },
    fathomed/.style={
        state,
        draw=red!80,
        fill=red!12,
        dashed
    },
    optimal/.style={
        state,
        draw=green!60!black,
        fill=green!10,
        ultra thick
    },
    suboptimal/.style={
        state,
        draw=black!60,
        fill=gray!5,
        dashed
    },
    info/.style={
        font=\tiny,
        text=gray
    }
]

\begin{scope}[on background layer]
    \fill[red!5] (0.5, 0.5) rectangle (2.5, -4.5);
    \node[red!70, font=\large, anchor=south west] at (0.5, -4.5) {$c=1$};
    
    \fill[blue!5] (2.5, 0.5) rectangle (4.5, -4.5);
    \node[blue!70, font=\large, anchor=south west] at (2.5, -4.5) {$c=2$};
\end{scope}

\foreach \i in {0,...,4} {
    \node[gray, font=\bfseries] at (\i, 1) {$\boldsymbol{i = \i}$};
}

\node[gray, font=\bfseries\footnotesize] at (1, 0.7) {$\boldsymbol{(6,15)}$};
\node[gray, font=\bfseries\footnotesize] at (2, 0.7) {$\boldsymbol{(4,8)}$};
\node[gray, font=\bfseries\footnotesize] at (3, 0.7) {$\boldsymbol{(2,3)}$};
\node[gray, font=\bfseries\footnotesize] at (4, 0.7) {$\boldsymbol{(1,1)}$};

\node[state] (s0) at (0,0) {
    $f=0$
    \nodepart{two} $0,0,0,0$
};

\node[fathomed] (s1_skip) at (1, 0) {
    $f=0$
    \nodepart{two} $0,0,0,0$
};
\node[red, font=\bfseries\tinysize, below=1pt of s1_skip] {F1: bound};

\node[state] (s1_pack) at (1, -2) {
    $f=15$
    \nodepart{two} $1,1,1,6$
};
\node[green!60!black, font=\bfseries\tinysize, below=1pt of s1_pack] {$LB=15$};

\node[state] (s2_skip) at (2, -1.5) {
    $f=15$
    \nodepart{two} $1,1,1,6$
};

\node[fathomed] (s2_pack) at (2, -3.5) {
    $f=23$
    \nodepart{two} $2,2,2,10$
};

\node[red, font=\bfseries\tinysize, below=1pt of s2_pack] {F2: Infeasible};

\draw[dotted, thick, gray] (2.5, 0.5) -- (2.5, -4.5) 
    node[fill=white, inner sep=1pt, font=\footnotesize, anchor=north, yshift= 1pt] at (2.5, 0.5) {$a \gets 0$};

\node[fathomed] (s3_skip) at (3, -1) {
    $f=15$
    \nodepart{two} $1,1,1,6$
};
\node[red, font=\bfseries\tinysize, below=1pt of s3_skip] {F1: bound};

\node[state] (s3_pack) at (3, -2.5) {
    $f=18$
    \nodepart{two} $2,1,1,8$
};
\node[green!60!black, font=\bfseries\tinysize, below=1pt of s3_pack] {$LB=18$};

\node[optimal] (s4_pack) at (4, -3) {
    $f=19$
    \nodepart{two} $3,2,2,9$
};
\node[green!50!black, font=\bfseries\tinysize, below=1pt of s4_pack] {OPTIMAL};

\node[suboptimal] (s4_skip) at (4, -1.8) {
    $f=18$
    \nodepart{two} $2,1,1,8$
};
\node[black!60, font=\bfseries\tinysize, below=1pt of s4_skip] {suboptimal};

\draw[->, thick, dashed, opacity=0.3] (s0) -- (s1_skip);
\draw[->, ultra thick, green!60!black] (s0) -- (s1_pack);

\draw[->, dashed, ultra thick, green!60!black] (s1_pack) --(s2_skip);
\draw[->, thick] (s1_pack) -- (s2_pack);

\draw[->, dashed, thick, opacity=0.5] (s2_skip) -- (s3_skip);
\draw[->, ultra thick, green!60!black] (s2_skip) -- (s3_pack);

\draw[->, ultra thick, green!60!black] (s3_pack) -- (s4_pack);
\draw[->, thick,  dashed, opacity=0.5] (s3_pack) -- (s4_skip);

\end{tikzpicture}
    
    \caption{Visual execution of $\firstDP$ on the instance from \Cref{fig:example1}.}
    \label{fig:firstDPexample}
\end{figure}

\subsection{Second dynamic program: iterating color by color}
\label{sec:secondDP}
The second dynamic programming algorithm, which we refer to as $\secondDP$, is based on a decomposition idea.
For each color $c \in \colors$, we determine the best possible packing for every possible number of items of that color and capacity usage, and then we solve a variant of the Multiple-Choice Knapsack Problem \citep[MCKP, see][]{KPP04li} where at most one packing can be selected for each color, while also adhering to the color constraints. 
To determine the relevant packings for a given color we utilize an `inner` dynamic program, and then solve the modified MCKP with an `outer` dynamic program.

The inner dynamic program determines the value
\begin{equation}\label{dp2:innerdpggoal}
    f_c(k,q) = \max\{\profit(S) \mid S \subseteq \items_c, \, \weight(S) = q, \, |S| = k\}, 
\end{equation}
which is the maximum profit obtainable if we solve a regular KP restricted to items in a given color $c \in \colors$, using exactly $q \in \{0,1,\dots, \capacity\}$ capacity, and packing exactly $k \in \{0,1,\dots, |\items_c|\} $ items.
If no such packing exists for some $(k,q)$, we set $f_c(k,q) = - \infty$.
The values for all feasible $k \in \{0,1,\dots, |\items_c|\}$ and $q \in \{0,1,\dots,\capacity\}$ can be determined by dynamic programming in at most $\mathcal{O}(\capacity \, |\items_c|^2)$ operations per color (see, e.g., \cite{KPP04li}). %

For the outer dynamic program, we want to determine the optimal CKP solution value if we only consider the first \(c\) of the $m$ colors, pack exactly $t$ items, of which $d$ are in each dominant color, and use exactly $q$ capacity.
Let $\mathcal{S}_c = \{(k,q) \mid f_c(k,q) > -\infty\}$ denote the combinations of number of items $k$ and capacity consumption $q$ for which there exists a feasible packing in color $c \in \colors$.
Formally, we want to determine
\begin{equation}\label{dp2outergoal}
\begin{aligned}
    f(c,t,d,q) = \max \Bigg\{\sum_{c' = 1}^c f_{c'}(k_{c'}, q_{c'}) \, \big| \, & \bigg((k_1,q_1), (k_2,q_2), \dots, (k_c,q_c)\bigg) \in \mathcal{S}_1 \times \mathcal{S}_2 \times \dots \times \mathcal{S}_c,\\[-1ex]
    & \sum_{c' = 1}^c k_{c'} = t, \, \max_{c' \in \{1,2,\dots,c\}} \, k_{c'} = d, \, \sum_{c' = 1}^c q_{c'} = q \Bigg\}
\end{aligned}
\end{equation}
for all \(c \in \{0, 1, \dots,m\}\), \(q \in \{0, 1, \dots,\capacity\}\), \(t \in \{0, 1, \dots,n\}\), and \(d \in \{0, 1, \dots,n\}\). We set \(f(c,q,t,d)\) to \(-\infty\) if no solution satisfies these conditions.
Then, we obtain the value of an optimal CKP solution as $\max_{t, d \in \{1,2,\dots,n\}, \, 2d \leq t + 1, \, q \in \{1,2,\dots,\capacity\}} f(m,t,d,q)$.

As for $\firstDP$, we explain this algorithm as it would be implemented instead of stating the corresponding recursion.
See~\Cref{seconddp} for a description in pseudo-code.
\begin{algorithm}
    \caption{$\secondDP$ algorithm for the CKP }
    \label{seconddp}
    \begin{algorithmic}[1]
        \Function{$\secondDP$}{$m,n, \boldsymbol \kappa, \boldsymbol\profit , \boldsymbol\weight, \capacity$}
        \State $f(c,t,d,q) \gets -\infty$ for all $c \in \{0, 1,\dots,m\}$, $t,d \in \{0,1,\dots,n\}$, $q \in \{0,1,\dots,\capacity\} $
        \State $f(0,0,0,0) \gets 0$ \Comment{Initialize starting state}
        \For{$c \in \{1,2,\dots,m\}$}
        \State determine $\mathcal{S}_c$ \Comment{Solve inner dynamic program}
        \For{$t,d,q$ such that $f(c-1,t,d,q) > -\infty$}
        \For{$(k,\Delta q) \in \mathcal{S}_c$ with  $q + \Delta q \leq \capacity $}
        \State $s_\mathrm{new} \gets (c, \, t + k, \, \max\{d,k\}, \, q+ \Delta q)$
        \State $f(s_\mathrm{new}) \gets \max\{f(s_\mathrm{new}), \, f(c-1,t,d,q) + f_{c}(k, \Delta q)\}$ \Comment{Try packing}
        \EndFor
        \EndFor
        \EndFor
        \EndFunction
    \end{algorithmic}
\end{algorithm}

For the initial state, we set \(f(0,0,0,0)\) to \(0\) and all other states to \(-\infty\).
Then we iterate through the colors and for the \(c\)-th color we extend each state $(c-1,t,d,q)$ with non-negative profit as follows:
for each feasible packing in this color using $\Delta q$ capacity and packing $k$ items (i.e., $(k,\Delta q) \in \mathcal{S}_c$) that still fits into the knapsack (i.e, $\Delta q \leq \capacity - q$), we determine the state we would reach if we selected this packing.
Specifically, the total number of items increases by $k$; the number of items in the dominant color is the maximum of what it was before and $k$, i.e., $d$ is only modified if we pack more items in the current color than in any previous color; the capacity usage is increased by $\Delta q$. 
Note that we also consider the null-transition with $\Delta q = k = 0$, i.e., packing no items of the current color.

\smallskip

\noindent The correctness of $\firstDP$ is formally stated in the following proposition.
\begin{proposition}
    The $\secondDP$ algorithm computes the value of $f(c,t,d,q)$ for $c \in \{0,1,\dots,m\}$, $t,d \in \{0,1,\dots,n\}$ and $q \in \{0,1,\dots,\capacity\}$ according to~\eqref{dp2outergoal}, if a corresponding solution exists, in at most $\mathcal{O}(\capacity^2 \, n^3)$ operations.
\end{proposition}
\begin{proof}
   For $c \in \{0,1,\dots,m\}$, $t,d \in \{0,1,\dots,n\}$, and $q \in \{0,1,\dots,\capacity\}$, let $f^*(c,t,d,q)$ denote the maximum profit obtainable by selecting items from the first $c$ colors such that exactly $t$ items are selected in total, the maximum number of selected items of any single color equals $d$, and the total weight equals $q$. If no such solution exists, we set $f^*(c,t,d,q) = -\infty$. 
The algorithm maintains a table $f(c,t,d,q)$. We prove that $f(c,t,d,q) = f^*(c,t,d,q)$ for all $c,t,d,q$ by induction on $c$.

\smallskip

For $c=0$, no colors are considered.
The only feasible solution selects no items, has zero weight, and $d=0$. Hence, $f^*(0,0,0,0)= 0$, and all other states are infeasible. 
The algorithm initializes exactly this state and assigns $-\infty$ to all others, so the claim holds.

\smallskip

We now assume that $f(c-1,t,d,q) = f^*(c-1,t,d,q)$ holds for all states. Consider an optimal solution corresponding to $f^*(c,t',d',q')$.
Such a solution selects a packing of color $c$ consisting of $k$ items with total weight $\Delta q$, where $(k,\Delta q) \in \mathcal{S}_c$.
The remaining items form a feasible solution over the first $c-1$ colors with parameters $t = t' - k$, $q = q' - \Delta q$, and $d = \max_{c' < c} k_{c'}(S)$. 
The value of $d$ of the full solution is therefore $d' = \max\{d,k\}$. \\
By optimality, the partial solution over the first $c-1$ colors must be optimal for the parameters $(c-1,t,d,q)$; otherwise, replacing it with a better one would increase the total profit. 
By the induction hypothesis, its profit equals $f(c-1,t,d,q)$. 
Moreover, the packing chosen for color $c$ must be optimal among all packings of $k$ items with weight $\Delta q$; otherwise, it could be replaced by a strictly better packing, contradicting the correctness of the inner dynamic program. \\
Thus, the total profit of the solution equals $f(c-1,t,d,q) + f_c(k,\Delta q)$,
and the algorithm considers this transition and updates the state
$(c,\, t+k,\, \max\{d,k\},\, q+\Delta q)$. Since the algorithm maximizes over all feasible predecessors and all feasible packings in
$\mathcal{S}_c$, it follows that $f(c,t',d',q') = f^*(c,t',d',q')$.
This completes the inductive proof.

\smallskip
The total running time of $\secondDP$ is the sum of the costs of the inner and outer dynamic programs. 
For a fixed color $c$, the inner dynamic program computes the table $f_c(k,q)$ in $\mathcal{O}(|\items_c|^2 \capacity)$ time. 
Summing over all colors and observing that $\sum_{c=1}^m |\items_c|^2 \le \Bigl(\sum_{c=1}^m |\items_c|\Bigr)^2 = n^2$, the total time for the inner phase is $\mathcal{O}(n^2 \capacity)$.
The outer dynamic program consists instead of $m$ stages. 
At each stage $c$, the state space $(t,d,q)$ has size $\mathcal{O}(n^2 \capacity)$. 
For each state, the algorithm iterates over all feasible packings in $\mathcal{S}_c$, whose size is bounded by $\mathcal{O}(|\items_c| \capacity)$. Thus, the time required at stage $c$ is $\mathcal{O}(n^2 \capacity \cdot |\items_c| \capacity) = \mathcal{O}(n^2 \capacity^2 |\items_c|)$.
Summing over all colors yields a total time complexity of
\begin{equation*}
    \mathcal{O}\left(n^2 \capacity^2 \sum_{c=1}^m |\items_c|\right)
= \mathcal{O}(n^3 \capacity^2).
\end{equation*}
\end{proof}

\subsection{Improvements for the second dynamic program}
Most of the improvements from $\firstDP$ can also be applied to $\secondDP$.
Specifically, we also use a sparse data structure and do the following reductions:
\begin{enumerate}
    \item[i)] For two states $s_1 = (i,t,d_1,q_1)$ and $s_2 = (i, t, d_2, q_2)$ with $q_1 \leq q_2$, $d_1 \leq d_2$ and $f(s_1) \geq f(s_2)$, it is possible to discard $s_2$.
    \item[ii)] For a state $(i,t,d,q)$ with $2d \leq t + 1$, we can set $d$ to $0$ because the dominant color among those considered so far is not critical and cannot become so anymore. 
\end{enumerate}

We also fathom states that cannot become feasible or optimal with the remaining items. 
Specifically, let $\bar{\profit}(i, q)$ and $\bar{n}(i,q)$ be the best possible profit and highest possible number of items that can be packed if we only consider colors from the $i$-th and using at most $q$ units of capacity. 
Then, we fathom a state $s = (i,t,d,q)$ if $f(s) + \bar{\profit}(i+1,\capacity-q) < LB$ for some primal bound $LB$ or if $2d > t + 1 + \bar{n}(i+1,\capacity-q)$.
We use the same criteria to discard states in the inner DP that can never be part of a feasible or optimal solution when considering the items of all other colors.

\subsection{Example for $\secondDP$}
\Cref{fig:firstDPexample} shows the execution of $\secondDP$ on the instance introduced in \Cref{fig:example1}, with $n=4$, $m =2$, and $b=10$.
In the figure, each node represents a generated state identified by the tuple $s = (t, d, q)$, with the associated profit $f(s)$ displayed above the tuple.
The nodes are divided into layers, each corresponding to a different stage of the algorithm (and, therefore, to a different color). 
Solid arrows correspond to the decision of packing at least one item of the current color, while dashed arrows represent the choice of packing none of its items. 
Above each arc we show the set of items packed in the next color and the corresponding profit.
In the upper section of each layer, the figure shows the packing options returned by the 'inner' dynamic program restricted to the sole items of that color.
For each option, we list the corresponding set of items $S$, number of items $k = |S|$, capacity usage $\Delta q$ and profit $f_c(k, \Delta q)$. 
All the remaining elements of the figure have the same meaning as in \Cref{fig:firstDPexample}.

\begin{figure}[htbp]
    \centering
    \begin{tikzpicture}[
    font=\scriptsize,
    >=Stealth,
    x=3.8cm, 
    y=1.3cm,
    state/.style={
        rectangle split,
        rectangle split parts=2,
        draw=black!80,
        rounded corners=2pt,
        fill=white,
        align=center,
        inner sep=3pt,
        text width=1.4cm
    },
    fathomed/.style={
        state,
        draw=red!80,
        fill=red!12,
        dashed
    },
    optimal/.style={
        state,
        draw=green!60!black,
        fill=green!10,
        ultra thick
    },
    suboptimal/.style={
        state,
        draw=black!60,
        fill=gray!5,
        dashed
    },
    inner_dp_box/.style={
        draw=gray!40,
        fill=white,
        dashed,
        rounded corners,
        font=\tiny,
        align=left,
        inner sep=4pt
    }
]

\begin{scope}[on background layer]
    \fill[red!5] (0.2, 1.5) rectangle (1.8, -6.5);
    
    \fill[blue!5] (1.8, 1.5) rectangle (3.4, -6.5);
\end{scope}

\draw[dotted, thick, gray] (1.8, 1.5) -- (1.8, -6.5);

\node[gray, font=\bfseries] at (-0.6, 1.8) {$\boldsymbol{c=0}$};
\node[gray, font=\bfseries] at (1, 1.8) {$\boldsymbol{c=1}$};
\node[gray, font=\bfseries] at (2.6, 1.8) {$\boldsymbol{c=2}$};

\node[inner_dp_box,scale=0.8, anchor=north] at (1,1.4)
{
\begin{tabular}{cccc}
     $S$& $k$ & $\Delta q$ & $f_1(k, \Delta q)$ \\
     \hline 
     $\emptyset$ & 0 & 0 & 0 \\
     $\{1\}$ &  1 & 6 & 15 \\
     $\{2\}$ & 1 & 4 & 8 \\
      $\{1,2\}$ & 2 & 10 & 23
\end{tabular}
};

\node[inner_dp_box,scale=0.8, anchor=north] at (2.6,1.4)
{
\begin{tabular}{cccc}
     $S$& $k$ & $\Delta q$ & $f_2(k, \Delta q)$ \\
     \hline 
     $\emptyset$ &0 & 0 & 0 \\
     $\{3\}$ &  1   & 2 & 3 \\
     $\{4\}$ & 1    & 1 & 1 \\
      $\{3,4\}$ & 2 & 3 &4 
\end{tabular}
};

\node[state] (s0) at (-0.6,-3) {
    $f=0$
    \nodepart{two} $0, 0, 0$
};

\node[fathomed] (s1_both) at (1, -5.25) {
    $f=23$
    \nodepart{two} $2, 2, 10$
};
\node[red, font=\bfseries\tinysize, below=1pt of s1_both] {F2: Infeasible};

\node[state] (s1_opt) at (1, -2.25) {
    $f=15$
    \nodepart{two} $1, 1, 6$
};
\node[green!50!black, font=\bfseries\tinysize, below=1pt of s1_opt] {$LB=15$};

\node[fathomed] (s1_item2) at (1, -3.75) {
    $f=8$
    \nodepart{two} $1, 1, 4$
};
\node[red, font=\bfseries\tinysize, below=1pt of s1_item2] {F1: Bound};

\node[fathomed] (s1_empty) at (1, -0.75) {
    $f=0$
    \nodepart{two} $0, 0, 0$
};
\node[red, font=\bfseries\tinysize, below=1pt of s1_empty] {F1: Bound};

\draw[->, thick, opacity=0.5] (s0) -- node[sloped, above] {$\{1,2\}\to 23$} (s1_both);
\draw[->, ultra thick, green!60!black] (s0) -- node[sloped,above] {$\{1\} \to 15$} (s1_opt);
\draw[->, thick, opacity=0.5] (s0) -- node[sloped,above]{$\{2\} \to 8$} (s1_item2);
\draw[->, dashed, thick, opacity=0.5] (s0) -- node[sloped,above]{$\emptyset \to 0$}(s1_empty);

\node[optimal] (s2_both) at (2.6, -5.25) {
    $f=19$
    \nodepart{two} $3, 2, 9$
};
\node[green!50!black, font=\bfseries\tinysize, below=1pt of s2_both] {OPTIMAL};

\node[suboptimal] (s2_i3) at (2.6, -2.25) {
    $f=18$
    \nodepart{two} $2, 1, 8$
};
\node[black!60, font=\tinysize, below=1pt of s2_i3] {suboptimal};

\node[suboptimal] (s2_i4) at (2.6, -3.75) {
    $f=16$
    \nodepart{two} $2, 1, 7$
};
\node[black!60, font=\tinysize, below=1pt of s2_i4] {suboptimal};

\node[suboptimal] (s2_empty) at (2.6, -0.75) {
    $f=15$
    \nodepart{two} $1, 1, 6$
};
\node[black!60, font=\tinysize, below=1pt of s2_empty] {suboptimal};

\draw[->, ultra thick, green!60!black] (s1_opt) -- node[above,sloped]{$\{3,4\} \to 4$} (s2_both);
\draw[->, thick, opacity=0.5] (s1_opt) --        node[above,sloped] {$\{3\} \to 3 $} (s2_i3);
\draw[->, thick, opacity=0.5] (s1_opt) --        node[above,sloped] {$\{4\} \to 1 $}  (s2_i4);
\draw[->, dashed, thick, opacity=0.5] (s1_opt) --node[above,sloped] {$\emptyset \to 0 $} (s2_empty);

\end{tikzpicture}
    
    \caption{Visual execution of $\secondDP$ on the instance from \Cref{fig:example1}.}
    \label{fig:secondDPexample}
\end{figure}

\section{Computational experiments} 
\label{sec:computationals}

In this section, we present the results of our computational experiments, aimed at evaluating the performance of the two exact Dynamic Programming algorithms for the CKP, $\firstDP$ and $\secondDP$, introduced in Sections \ref{sec:firstDP} and \ref{sec:secondDP}, respectively. 
Our experiments pursue two main goals: first, to assess the computational effectiveness of the proposed dynamic programming approaches by comparing them against the direct solution of the $\formCKP$ formulation using state-of-the-art MIP solvers; second, to demonstrate the practical efficiency of the algorithms on instances covering a wide range of structural characteristics.

To the best of our knowledge, the CKP is a novel generalization of the Knapsack Problem that has not been previously studied in the literature (see \Cref{sec:lit}). Consequently, no publicly available benchmark instances exist. For this reason, we introduce a comprehensive benchmark of instances derived from the pricing subproblems of the Colored Bin Packing Problem (CBPP), which we describe in \Cref{sec:instances}.

\subsection{Library of benchmark CKP instances}
\label{sec:instances}

In this section, we introduce the library of benchmark instances specifically designed for the CKP. This testbed aims to provide a diverse and representative collection of instances, varying in size and structural characteristics, with the goal of comparing on such dataset the proposed methods and obtaining insights about how different instance features influence computational performance.

As discussed in \Cref{sec:lit}, the CKP arises as pricing problem of the CBPP (see \cite{BorgesSM24}). Thus, we decided to generate CKP instances by solving the root node of CBPP instances by column generation. As a result, we obtained a large set of CKP instances for each CBPP instance. All CKP instances corresponding to the same CBPP instance share the same number of items, the items sizes, the item colors, and the capacity of the knapsack. In contrast, the profits of the items differ as they stem from the dual solutions of the master problem and are used to compute the reduced costs of columns. Since $\boldsymbol \profit \in \mathbb{Z}^n$, we scaled the dual variable values and set $p_i = \lfloor 10^5\pi_i \rfloor$ for $i \in \items$ and $\pi_i$ being the corresponding dual variable value. For our computational campaign, we selected a subset of the instances proposed by \citet{BorgesSM24}. 
Specifically, we used the uniform randomly generated instances, the randomly generated instances with zipf distribution of colors, and the ANI instances with $201$ items and $2$ colors, which were taken from the BPP literature. For the specific used parameters we refer to \citet{BorgesSM24}.

Since our approaches use bounds obtained by solving a Knapsack Problem for each instance, we check if the returned KP solution is feasible for the CKP as well. Instances for which the optimal solution to the standard Knapsack Problem was already feasible for the CKP (i.e., it satisfied the color constraints) were removed from the original set because they can be easily solved using any state-of-the-art knapsack solver. The results we present in the following section are therefore just based on instances where our implementation did not find a knapsack solution that is feasible for the CKP.

\subsection{Computational performance}
\label{sec:performance}

In this section, we assess the computational performance of our proposed dynamic programming algorithms, $\firstDP$ and $\secondDP$, by comparing them against the direct solution of the $\formCKP$ formulation using two state-of-the-art MIP solvers: the commercial solver {\tt IBM ILOG CPLEX} \citep{cplex} and the free and open-source solver {\tt SCIP} \citep{SCIPOptSuite10}.
All tests were conducted on Debian 11 computing nodes equipped with two Intel Xeon L5630 processors with at least 16 GB of DDR3 RAM.
We used the \texttt{C++} interfaces of \texttt{IBM ILOG CPLEX v22.1.0} and \texttt{SCIP v10.0.0}, running in single-thread mode with default settings. Moreover, we implemented our algorithms in \texttt{C++}.

\medskip

Table~\ref{tab:uniform_summary_table} summarizes the results on the set of instances of the uniform family. The table reports the performance metrics on a large set of uniformly randomly generated instances. The instances are categorized by the number of items $n$, the capacity $\capacity$, the number of color classes $m$, and the weight generation interval $W$ ($W_1 = [0.1, 0.8]$ and $W_2 = [0.01, 0.25]$ relative to capacity).
The column ``\# inst.'' reports the number of instances in the test set without the trivial (i.e., KP-solveable) instances.
The column ``\% rem.'' indicates the percentage of of overall instances that were removed for being trivial.
Finally, the last four columns report the average solving times in milliseconds for the four approaches on these remaining instances. Bold values indicate the lowest average computational time.

The data in Table \ref{tab:uniform_summary_table} highlights several key insights: first, the ``\% rem.'' column reveals a strong inverse correlation between the number of color classes ($m$) and the restrictiveness of the color constraints. When $m$ is high (e.g., $m=15$), the probability of selecting consecutive items of the same color is lower. Consequently, for instances with $m=15$, up to $99.9\%$ of the generated instances were removed because the standard knapsack solution was already valid. Conversely, for $m=2$, the color constraint is highly restrictive, with removal rates dropping to approximately 20--35\%, which suggests that, for the majority of these instances, the color constraints are binding.

Regarding computational performance, the proposed DP algorithms significantly outperform the general-purpose solvers on the set of non-trivial uniform instances. $\firstDP$ is the fastest approach in 29 out of 30 classes, with an overall average time of \qty{14}{\milli\second}. This is orders of magnitude faster than {\tt CPLEX} (\qty{1425}{\milli\second}) and {\tt SCIP} (\qty{4294}{\milli\second}).
The difficulty of the instances for the MIP solvers appears to be driven by the weight structure ($W_1$ vs. $W_2$). Instances with $W_1$ (weights in $[0.1, 0.8]$) are harder for branch-and-bound solvers, with {\tt CPLEX} and {\tt SCIP} taking on average multiple seconds for instances with $n=1,000$. In contrast, our DP algorithms remain extremely efficient on these inputs, often solving them in under \qty{5}{ms}. 
The only case where {\tt CPLEX} is competitive with $\firstDP$ is for $n=1,000$ with $m=2$ and $W_2$ weights, where it achieves a marginally better time (\qty{62.7}{\milli\second} vs \qty{63.8}{\milli\second}).
However, $\firstDP$ remains efficient across the entire range of parameters, which proves it is the most efficient approach for solving instances of this family.

\begin{table}
\renewcommand
\arraystretch{0.65}
\centering
    \caption{Results for the uniform randomly generated instances.
    For each class of instances we give the number of items $n$, the capacity $b$, the number of colors $m$, and the interval $W$ from which the item size (as a proportion of $b$) is sampled ($W_1 = [0.1,0.8], W_2 = [0.01,0.25]$.
    We further state the number of instances after removal, what percentage of them were removed because the standard knapsack solution was feasible for the CKP, and the average running time in milliseconds for each approach.}
    
    \medskip
    
    \begin{tabular}{llllrrrrrr}
\toprule
\multicolumn{4}{l}{Instance class} & &&\multicolumn{4}{r}{Solving time (ms)}  \\[1mm]
\cmidrule{1-4} \cmidrule(l{20pt}){7-10}
$n$ & $\capacity$ & $m$ & $W$ & \# inst. & \% rem. &\hspace{1cm} ${\tt DP_1}$ & ${\tt DP_2}$ & ${\tt CPLEX}$ & ${\tt SCIP}$ \\[1mm]
\cmidrule{1-6} \cmidrule(l{20pt}){7-10}
&&\\[-1.3ex]
\multirow[t]{12}{*}{\num{300}} & \multirow[t]{6}{*}{\num{500}} & \multirow[t]{2}{*}{\num{2}} & $W_1$ & \num{5727} & \num{35.1} & \textbf{\num{1.5}} & \num{3.0} & \num{186.6} & \num{1130.4} \\
 &  &  & $W_2$ & \num{6325} & \num{20.5} & \textbf{\num{21.2}} & \num{64.9} & \num{34.0} & \num{292.7} \\[0.7ex]
 &  & \multirow[t]{2}{*}{\num{7}} & $W_1$ & \num{2289} & \num{79.1} & \textbf{\num{1.2}} & \num{4.5} & \num{1132.0} & \num{5118.7} \\
 &  &  & $W_2$ & \num{112} & \num{98.6} & \textbf{\num{1.3}} & \num{16.0} & \num{37.7} & \num{262.2} \\[0.75ex]
 &  & \multirow[t]{2}{*}{\num{15}} & $W_1$ & \num{1458} & \num{86.0} & \textbf{\num{1.2}} & \num{2.9} & \num{1545.3} & \num{7497.6} \\
 &  &  & $W_2$ & \num{14} & \num{99.8} & \textbf{\num{1.2}} & \num{26.0} & \num{27.4} & \num{187.1} \\[1.5ex]
 & \multirow[t]{6}{*}{\num{750}} & \multirow[t]{2}{*}{\num{2}} & $W_1$ & \num{6533} & \num{32.3} & \textbf{\num{2.1}} & \num{5.1} & \num{296.7} & \num{1503.1} \\
 &  &  & $W_2$ & \num{6641} & \num{19.8} & \textbf{\num{19.9}} & \num{83.0} & \num{42.5} & \num{358.9} \\[0.7ex]
 &  & \multirow[t]{2}{*}{\num{7}} & $W_1$ & \num{2175} & \num{79.3} & \textbf{\num{1.8}} & \num{5.9} & \num{1254.5} & \num{5494.6} \\
 &  &  & $W_2$ & \num{102} & \num{98.8} & \textbf{\num{2.0}} & \num{35.8} & \num{39.7} & \num{331.3} \\[0.7ex]
 &  & \multirow[t]{2}{*}{\num{15}} & $W_1$ & \num{1404} & \num{86.6} & \textbf{\num{1.8}} & \num{4.2} & \num{1493.3} & \num{6802.6} \\
 &  &  & $W_2$ & \num{8} & \num{99.9} & \textbf{\num{1.9}} & \num{19.9} & \num{25.9} & \num{135.6} \\[2ex]
\multirow[t]{18}{*}{\num{500}} & \multirow[t]{6}{*}{\num{500}} & \multirow[t]{2}{*}{\num{2}} & $W_1$ & \num{12014} & \num{33.0} & \textbf{\num{2.4}} & \num{5.5} & \num{442.0} & \num{2400.9} \\
 &  &  & $W_2$ & \num{10815} & \num{22.2} & \textbf{\num{21.2}} & \num{62.8} & \num{43.3} & \num{413.0} \\[0.7ex]
 &  & \multirow[t]{2}{*}{\num{7}} & $W_1$ & \num{4108} & \num{78.8} & \textbf{\num{2.1}} & \num{7.7} & \num{6414.8} & \num{10906.5} \\
 &  &  & $W_2$ & \num{176} & \num{98.8} & \textbf{\num{2.2}} & \num{25.8} & \num{48.7} & \num{426.2} \\[0.7ex]
 &  & \multirow[t]{2}{*}{\num{15}} & $W_1$ & \num{2425} & \num{86.6} & \textbf{\num{2.1}} & \num{6.8} & \num{9021.6} & \num{19130.9} \\
 &  &  & $W_2$ & \num{19} & \num{99.9} & \textbf{\num{2.1}} & \num{38.5} & \num{50.7} & \num{359.2} \\[1.5ex]
 & \multirow[t]{6}{*}{\num{750}} & \multirow[t]{2}{*}{\num{2}} & $W_1$ & \num{12830} & \num{33.2} & \textbf{\num{4.1}} & \num{8.8} & \num{886.2} & \num{2934.2} \\
 &  &  & $W_2$ & \num{11140} & \num{23.7} & \textbf{\num{22.8}} & \num{95.7} & \num{53.3} & \num{493.0} \\[0.7ex]
 &  & \multirow[t]{2}{*}{\num{7}} & $W_1$ & \num{4347} & \num{78.1} & \textbf{\num{3.5}} & \num{12.4} & \num{3756.6} & \num{13897.3} \\
 &  &  & $W_2$ & \num{185} & \num{98.7} & \textbf{\num{3.6}} & \num{46.5} & \num{59.5} & \num{547.7} \\[0.7ex]
 &  & \multirow[t]{2}{*}{\num{15}} & $W_1$ & \num{2593} & \num{87.3} & \textbf{\num{3.4}} & \num{10.5} & \num{6381.5} & \num{18692.0} \\
 &  &  & $W_2$ & \num{18} & \num{99.9} & \textbf{\num{3.5}} & \num{58.7} & \num{64.5} & \num{593.1} \\[1.5ex]
 & \multirow[t]{6}{*}{\num{1000}} & \multirow[t]{2}{*}{\num{2}} & $W_1$ & \num{13104} & \num{32.0} & \textbf{\num{5.4}} & \num{11.9} & \num{1113.6} & \num{3111.3} \\
 &  &  & $W_2$ & \num{12095} & \num{21.8} & \num{63.8} & \num{214.8} & \textbf{\num{62.7}} & \num{560.2} \\[0.7ex]
 &  & \multirow[t]{2}{*}{\num{7}} & $W_1$ & \num{5078} & \num{74.4} & \textbf{\num{4.9}} & \num{17.8} & \num{5536.2} & \num{15336.5} \\
 &  &  & $W_2$ & \num{162} & \num{99.0} & \textbf{\num{4.9}} & \num{69.0} & \num{65.3} & \num{567.3} \\[0.7ex]
 &  & \multirow[t]{2}{*}{\num{15}} & $W_1$ & \num{3634} & \num{81.9} & \textbf{\num{4.8}} & \num{14.5} & \num{7245.5} & \num{22216.5} \\
 &  &  & $W_2$ & \num{22} & \num{99.9} & \textbf{\num{4.8}} & \num{109.0} & \num{49.9} & \num{422.8} \\[1.5ex]
\cmidrule{1-6} \cmidrule(l{20pt}){7-10}
Tot/Avg &  &  &  & \num{127553} & \num{69.4} & \textbf{\num{14.0}} & \num{47.3} & \num{1425.8} & \num{4294.1} \\
\bottomrule
\end{tabular}

    \label{tab:uniform_summary_table}
\end{table}

\medskip

Table \ref{tab:zipf_summary_table} summarizes instead the results on the benchmark set of instances with zipf distribution of colors. Similarly to \Cref{tab:uniform_summary_table}, this table reports the instance parameters ($n$, $\capacity$), the number of  instances before filtering (``\# inst.''), the percentage of removed instances (``\% rem.''), and the average solving times for the four competing algorithms.
The results in Table~\ref{tab:zipf_summary_table} present a different landscape compared to the uniform distribution case. Here, {\tt CPLEX} consistently outperforms all other methods, achieving the lowest average solving times across all instance classes, with an overall average slightly above \qty{71}{\milli\second}.
In contrast, $\firstDP$ and $\secondDP$ require significantly more time, with an average solving time of roughly 1 second per instance.
This performance shift can be attributed to the structural properties induced by the color distribution. When items are concentrated in a few dominant color classes, the effective state space for the DP algorithms may remain large, as there are many items of the same color that can potentially be packed. On the other hand, the color constraints may become more binding in a branch-and-bound context. Since a large portion of items share the same color, the branching decisions quickly lead to infeasibility when attempting to pack many of them, allowing the solvers to prune the branching tree more aggressively.

Furthermore, the ``\% rem.'' column indicates that very few instances (only 2--7\%) were removed. This confirms that the standard knapsack solution is rarely feasible for these instances, implying that there is almost always an active color constraint. Nevertheless, it is worth noting that while slower than {\tt CPLEX}, the DP algorithms still maintain reasonable solving times (typically under 2 seconds), remaining a viable option even for this class of instances. Furthermore, its performance is still comparable to that of {\tt SCIP}, especially for smaller instances with $\capacity\leq500$: specifically, for the group of smallest instances ($n=300$ and $\capacity = 300$), $\firstDP$ and $\secondDP$ have an average solving time which is around half of that of {\tt SCIP}.

\begin{table}
\renewcommand
\arraystretch{0.65}
    \caption{Results for the randomly generated instances with zipf distribution of colors. Given are the parameters of each class of instances, how many instances they comprise (after removal), and how many instances were removed because the standard knapsack solution was feasible for the CKP. In addition, the average running times in milliseconds of each approach for each class of instances are shown. $n$ corresponds to the number of items, $b$ to the capacity.}
    \centering
    \begin{tabular}{llrrrrrr}
\toprule
\multicolumn{4}{l}{Instance class} & \multicolumn{4}{r}{Solving time (s)}  \\[1mm]
\cmidrule{1-2} \cmidrule(l{20pt}){5-8}
$n$ &$\capacity$  & \# inst. & \% rem. & \hspace{1cm}${\tt DP_1}$ & ${\tt DP_2}$ & ${\tt CPLEX}$ & ${\tt SCIP}$ \\[1mm]
\cmidrule{1-4} \cmidrule(l{20pt}){5-8}
&&\\[-2ex]
\multirow[t]{3}{*}{\num{300}} & \num{300} & \num{7012} & \num{7.0} & \num{173.7} & \num{128.2} & \textbf{\num{42.0}} & \num{313.7} \\
 & \num{500} & \num{7630} & \num{4.4} & \num{331.5} & \num{308.1} & \textbf{\num{53.0}} & \num{400.0} \\
 & \num{750} & \num{8399} & \num{3.4} & \num{719.5} & \num{885.4} & \textbf{\num{64.3}} & \num{472.4} \\[1ex]
\multirow[t]{3}{*}{\num{500}} & \num{300} & \num{12466} & \num{2.8} & \num{535.3} & \num{404.9} & \textbf{\num{62.9}} & \num{505.2} \\
 & \num{500} & \num{13227} & \num{3.5} & \num{1017.2} & \num{973.8} & \textbf{\num{80.3}} & \num{639.4} \\
 & \num{750} & \num{14585} & \num{2.1} & \num{2102.4} & \num{2523.8} & \textbf{\num{97.9}} & \num{785.4} \\[1ex]
\cmidrule{1-4} \cmidrule(l{20pt}){5-8}
Tot/Avg &  & \num{63319} & \num{3.6} & \num{956.8} & \num{1033.2} & \textbf{\num{71.3}} & \num{559.6} \\
\bottomrule
\end{tabular}

    \label{tab:zipf_summary_table}
\end{table}

\medskip

To conclude our experimental analysis on different families of instances, we also evaluated the algorithms on the ANI class, a dataset comprising \num{50513} instances after \SI{29.6}{\percent} were removed. The results confirm the trends observed for the uniformly generated instances. $\firstDP$ confirms to be the most performing algorithm, solving these instances with an average time of just \qty{11}{\milli\second}.
This performance is approximately $36$ times faster than {\tt CPLEX}, which requires an average time slightly above \qty{400}{\milli\second}, and orders of magnitude faster than {\tt SCIP} (\qty{1653}{\milli\second} on average).
$\secondDP$ remains competitive with an average time of \qty{80}{\milli\second}, though less efficient than $\firstDP$. 

Finally, Figure \ref{fig:pricing_times} illustrates the average solving time required by the four approaches ($\firstDP$, $\secondDP$, {\tt CPLEX}, and {\tt SCIP}) to solve the pricing subproblems generated during the solution of the CBPP at the root node of the branch-and-price tree. The $x$-axis represents the progress of the column generation procedure, expressed as the percentage of pricing problems solved, while the $y$-axis reports the average computational time taken by the different solution algorithms.

A clear increasing trend in difficulty is observable for all methods as the column generation progresses. This behavior is typical in branch-and-price algorithms: in the early iterations, the dual values often distribute unevenly, making it easier for solvers to find optimal columns or prune the search space. As the column generation converges, the dual variables stabilize, and the distinction between promising and non-promising items becomes less pronounced, typically requiring a more exhaustive search to certify optimality.
However, the impact of this increasing difficulty varies significantly between the approaches. The MIP solvers, particularly {\tt SCIP}, exhibit a more significant growth in solving times, which shows a greater sensitivity to the specific distribution of item profits. In contrast, the DP algorithms, particularly $\firstDP$, demonstrate a greater stability. This is consistent with the design of $\firstDP$: its computational complexity is indeed primarily determined by capacity and number of items, parameters whose value remains constant throughout the column generation process. Albeit the fathoming rules may become less effective as the item profit distribution varies, this does not significantly affect the performance of $\firstDP$, whose profile remains below that of all other solution approaches for almost the entire $x$-axis.

\begin{figure}
    \centering
    \input{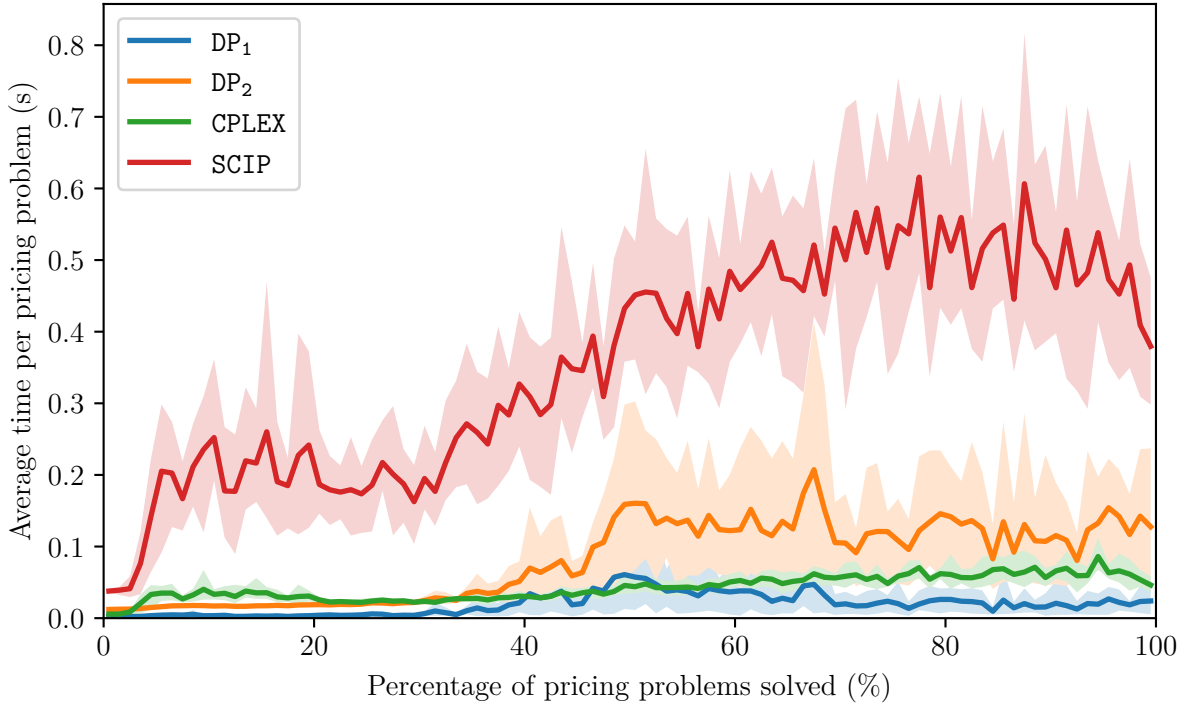}
    \caption{Evolution of the average solving time (in seconds) required by the four approaches to solve the pricing subproblems generated during the solution of the CBPP at the root node of the branch-and-price tree.} 
    \label{fig:pricing_times}
\end{figure}

\section{Conclusions}
\label{sec:conclusions}

In this paper, we introduced and analyzed the Colored Knapsack Problem (CKP), a novel generalization of the classical Knapsack Problem where items are partitioned into color classes and must be packed avoiding consecutive items of the same color. From a theoretical standpoint, we established the weak NP-hardness of the problem and analyzed its linear programming relaxation. We proved that an optimal solution to the relaxation always exists with at most two fractional items and demonstrated that it can be found in $\mathcal{O}(n)$ time, matching the complexity of the classical KP.

From the computational perspective, we instead proposed two pseudo-polynomial time dynamic programming algorithms, $\firstDP$ and $\secondDP$, with time complexities of $\mathcal{O}(bn^4)$ and $\mathcal{O}(b^{2}n^{3})$, respectively. We further enhanced these algorithms with specific dominance and fathoming rules to reduce the search space. Our computational experiments, performed on a newly introduced benchmark set of instance derived from the Colored Bin Packing Problem, showed that our dynamic programming approaches significantly outperform state-of-the-art general-purpose MIP solvers on most of the instance families.

Future research avenues naturally include the development of approximation algorithms or fully polynomial-time approximation schemes (FPTAS) for the CKP, as well as the design of combinatorial branch-and-bound algorithms which leverage the low computational effort required to find an optimal solution to the CKP linear relaxation.
A polyhedral study of the problem also appears promising, specifically to derive new valid inequalities or to extend classic cuts, such as cover cuts, to this knapsack variant.
Furthermore, given the close relationship with the Colored Bin Packing Problem, incorporating the proposed algorithms as pricing solvers within a branch-and-price framework for the CBPP constitutes a promising direction for future study.

\section*{Acknowledgments}
We would like to thank Fabio Furini for giving us the inspiration for this work, as well as Fabio Furini and Marco Lübbecke for discussions regarding the CKP and presentation of the results.

\section*{Statement on AI Usage in This Work}
We attempted to utilize commercially available LLMs like Gemini and ChatGPT to generate ideas for Dynamic Programming approaches to this problem but the proposed algorithms were either asymptotically slower than our ideas (e.g., using one resource per color to track the number of items packed in that color) or not correct. 
For solving the LP relaxation of the CKP we originally developed a cubic time algorithm.
Once we realized that the results of~\cite{Megiddo1984} in combination with the property that we can ``predict'' the dominant color in the LP may lead to a linear time algorithm, the same LLMs were used to find appropriate sources and develop the proof idea.
All text in this manuscript was written by us without any AI assistance and all citations were checked manually by us. 
We used GitHub Copilot for support in implementing our proposed algorithms and integer programming models as well as creating plots and tables. 
All AI-generated code was manually checked for accuracy by us.

\bibliographystyle{abbrvnat}
\bibliography{biblio}

\end{document}